\documentclass[12pt]{amsart}
\usepackage{hyperref,amssymb}

\theoremstyle{plain}
\newtheorem{theorem}{Theorem}[section]
\newtheorem{proposition}[theorem]{Proposition}
\newtheorem{corollary}[theorem]{Corollary}
\newtheorem{lemma}[theorem]{Lemma}

\theoremstyle{definition}
\newtheorem{remark}[theorem]{Remark}
\newtheorem{question}[theorem]{Question}

\newtheorem{example}[theorem]{Example}

\newcommand{\abs}[1]{\lvert#1\rvert}
\newcommand{\norm}[1]{\lVert#1\rVert}
\newcommand{\bigabs}[1]{\bigl\lvert#1\bigr\rvert}

\newcommand{\term}[1]{{\textit{\textbf{#1}}}}   % To introduce a term
\renewcommand{\mid}{\::\:}

\newcommand{\goesu}{\xrightarrow{\mathrm{u}}}	% u-convergence
\newcommand{\goesuo}{\xrightarrow{\mathrm{uo}}}	% uo-convergence

\DeclareSymbolFont{bbold}{U}{bbold}{m}{n}
\DeclareSymbolFontAlphabet{\mathbbold}{bbold}
\def\one{\mathbbold{1}}

\DeclareMathOperator{\Range}{Range}
\DeclareMathOperator{\Span}{span}

\DeclareMathOperator{\codim}{codim}

\renewcommand{\le}{\leqslant}
\renewcommand{\ge}{\geqslant}

\begin{document}

\title{Uniformly closed sublattices of finite codimension}

\author{Eugene Bilokopytov}
\email{bilokopy@ualberta.ca}

\author{Vladimir G. Troitsky}
\email{troitsky@ualberta.ca}

\address{Department of Mathematical and Statistical Sciences,
  University of Alberta, Edmonton, AB, T6G\,2G1, Canada.}

\thanks{The second author was supported by an NSERC grant.}
% 46B42 Banach lattices
% 46A40 Ordered topological linear spaces, vector lattices
% 46E05 Lattices of continuous, differentiable or analytic functions

\date{\today}

\begin{abstract}
  The paper investigates uniformly closed subspaces, sublattices, and
  ideals of finite codimension in Archimedean vector lattices. It is
  shown that every uniformly closed subspace (or sublattice) of finite
  codimension may be written as an intersection of uniformly closed
  subspaces (respectively, sublattices) of codimension one. Every
  uniformly closed sublattice of codimension $n$ contains a uniformly
  closed ideal of codimension at most $2n$. If the vector lattice is
  uniformly complete then every ideal of finite codimension is
  uniformly closed.  Results of the paper extend (and are
  motivated by) results of~\cite{Abramovich:90a,Abramovich:90b}, as
  well as Kakutani's characterization of closed sublattices of $C(K)$
  spaces.
\end{abstract}

\maketitle

\section{Introduction and preliminaries}

This paper merges two lines of investigation. The first line goes back
to the celebrated Krein-Kakutani Theorem \cite{Krein:40,Kakutani:41}
that every Archimedean vector lattice with a strong unit can be represented as a
norm dense sublattice of $C(K)$ for some compact Hausdorff space $K$.
Furthermore, Kakutani in~\cite{Kakutani:41} completely characterized
closed sublattices of $C(K)$ spaces.

The second line was initiated by Abramovich and Lipecki
in~\cite{Abramovich:90a,Abramovich:90b}, where they studied
sublattices and ideals of finite codimension in vector and Banach
lattices. They proved, in particular, that every vector lattice has
sublattices of all finite codimensions and that every finite
codimensional ideal in a Banach lattice is closed.

We start our paper by revisiting Kakutani's characterization of closed
sublattices of $C(K)$ spaces. We provide an easier proof of the
characterization; we also extend it from $C(K)$ to $C(\Omega)$, where $\Omega$
need not be compact. We characterize closed sublattices of Banach
sequence spaces.

In the rest of the paper, we investigate uniformly closed subspaces,
sublattices, and ideals of finite codimension in an arbitrary
Archimedean vector lattice $X$. By ``uniformly closed'', we mean
``closed with respect to relative uniform convergence''.
While~\cite{Abramovich:90a,Abramovich:90b} use algebraic
techniques based on prime ideals, our approach is to use
Krein-Kakutani Representation theorem to reduce the problems to the
case of $C(\Omega)$ spaces, and then use Kakutani's characterization of
closed sublattices there. In a certain sense, it is the same approach
because one can use the technique of prime ideals to prove
Krein-Kakutani Representation Theorem. However, using the theorem
makes proofs easier and more transparent. We extend several of the
results of~\cite{Abramovich:90a,Abramovich:90b} about topologically
closed sublattices in Banach and F-lattices to uniformly closed
sublattices in vector lattices.

For a linear functional $\varphi$ on $X$, we relate properties of
$\varphi$ to those of the one codimensional subspace
$\ker\varphi$. Namely, we show that $\varphi$ is positive or negative
if and only if $\ker\varphi$ is full, $\varphi$ is order bounded if and only if
$\ker\varphi$ is uniformly closed, and $\varphi$ is a difference of two
lattice homomorphisms if and only if $\ker\varphi$ is a uniformly closed
sublattice. We show that every uniformly closed subspace (or
sublattice) of finite codimension may be written as an intersection of
uniformly closed subspaces (respectively, sublattices) of codimension
one. We show that every uniformly closed sublattice of codimension $n$
contains a uniformly closed ideal of codimension at most $2n$. We
prove that if $X$ is uniformly complete then every ideal of finite
codimension is uniformly closed.

Throughout the paper, $X$ stands for an Archimedean vector
lattice. For background on vector lattices, we refer the reader
to~\cite{Aliprantis:06,Luxemburg:71,Meyer-Nieberg:91}.

\medskip

We will now provide a brief overview of standard linear algebra facts
on subspaces of finite codimension that will be used throughout the
paper.  Let $E$ be a vector space over $\mathbb R$. By a functional on
$E$ we always mean a linear functional.  We write $E'$ for the linear
dual of $E$. For a subset $B$ of $E'$, we write $B_\perp$ for the
pre-annihilator of $B$, i.e.,
$B_\perp=\bigl\{x\in E\mid\forall\varphi\in B\ \varphi(x)=0\bigr\}$.
For a subspace $F$ of $E$, we say that $F$ is of codimension $n$ in
$E$ if $\dim E/F=n$. Equivalently, $F\cap G=\{0\}$ and $F+G=E$ for
some subspace $G$ with $\dim G=n$. Equivalently, there exist linear
independent functionals $\varphi_1,\dots,\varphi_n$ in $E'$ such that
$F=\bigcap_{i=1}^n\ker\varphi_i=\{\varphi_1,\dots,\varphi_n\}_\perp$. It
follows that if $\codim F=n$ and $G\subseteq E$ is another subspace
then the codimension of $F\cap G$ in $G$ is at most $n$.  If
$G\subseteq F$ then the codimension of $F$ in $E$ equals the
codimension of $F/G$ in $E/G$. We write $F^\perp$ for the annihilator
of $F$ in $E'$, i.e., for the set of all functionals in $E'$ that
vanish on $F$. However, when $E$ is a topological vector space, one
often writes $F^\perp$ for the annihilator of $F$ in the topological
dual $E^*$ of $E$, i.e., the set of all functionals in $E^*$ that
vanish on $F$. Every time we use the $F^\perp$ notation, it will be
clear from the context whether we take the annihilator in $E'$ or in
$E^*$.  In particular, let $F$ be a closed subspace of a normed space
$E$, then the functionals $\varphi_1,\dots,\varphi_n$ above may be
chosen in the topological dual $E^*$; in this case,
$\codim F=\dim F^\perp$, where $F^\perp$ is the annihilator of $F$ in
$E^*$.

\begin{lemma}\label{codim-compl}
  Let $F$ be a closed subspace of codimension $n$ in a normed space
  $E$. Then the closure $\overline{F}$ of $F$ in the completion
  $\overline{E}$ of $E$ has codimension $n$.
\end{lemma}

\begin{proof}
  The annihilator of $F$ in $E^*$ has dimension $n$. The canonical
  isometric isomorphism between $E^*$ and $\overline{E}^*$ that sends
  every functional in $E^*$ to its unique extension in
  $\overline{E}^*$ maps this annihilator to the annihilator of $F$ in
  $\overline{E}^*$; hence the latter also has dimension $n$. It
  follows that $\codim(F^\perp)_\perp=n$, where both operations
  are performed with respect to the dual pair
  $(\overline{E},\overline{E}^*)$. Finally, recall that
  $(F^\perp)_\perp=\overline{F}$.
\end{proof}

\section{Closed sublattices of $C(\Omega)$}
\label{CK-cl-sublat}

In this section, we revisit Kakutani's characterization of closed
sublattices of $C(K)$ spaces from ~\cite{Kakutani:41}.
Throughout this section, $\Omega$ stands for a completely regular Hausdorff
topological space (also known as a Tychonoff space), which is exactly
the class of Hausdorff spaces where the conclusion of Urysohn's
lemma holds. Recall that locally compact Hausdorff spaces and
normal spaces are completely regular.

We equip the space $C(\Omega)$ of all real valued continuous functions
on $\Omega$ with the compact-open topology, i.e., the topology of
uniform convergence on compact sets.  In the special case of a compact
space $K$, this topology agrees with the supremum norm topology on
$C(K)$.  Note that for every $t\in\Omega$, the point evaluation
functional $\delta_t(f)=f(t)$ is a continuous positive linear
functional on $C(\Omega)$, i.e., $\delta_f\in C(\Omega)^*_+$.  We will
characterize closed sublattices of $C(\Omega)$, as well as describe the
closure of a sublattice of $C(\Omega)$.

We start with some motivation and examples. It can be easily verified
that the set of all functions $f$ in $C[0,1]$ satisfying $f(1)=2f(0)$
forms a closed sublattice. More generally, fix $s\ne t$ in $\Omega$
and a non-negative scalar $\alpha$, and put
$$Y=\bigl\{f\in C(\Omega)\mid f(s)=\alpha f(t)\bigr\};$$ then $Y$ is a
closed sublattice of $C(\Omega)$. Note that $Y=\ker\mu$ where
$\mu\in C(\Omega)^*$ is given by $\mu=\delta_s-\alpha\delta_t$. In
case $\alpha=0$, we just have $\mu=\delta_s$; then $Y$ is
actually an ideal and $t$ is irrelevant.

In the preceding paragraph, $Y$ was determined by a single constraint
$f(s)=\alpha f(t)$. We may generalize this construction to an
arbitrary family of constraints as follows. Consider a collection of
triples $(s,t,\alpha)$, where $s$ and $t$ are two distinct points in
$\Omega$ and $\alpha\ge 0$, and let $Y$ be the set of all functions
$f$ in $C(\Omega)$ such that $f(s)=\alpha f(t)$ for every triple
$(s,t,\alpha)$ in the collection. There is an alternative way to
describe $Y$: let $\mathcal M$ be a family of linear functionals on
$C(\Omega)$ of the form $\mu=\delta_s-\alpha\delta_t$, where the
triple $(s,t,\alpha)$ is in the collection; then
$Y=\bigcap_{\mu\in\mathcal M}\ker\mu=\mathcal M_\perp$. Again, it is
easy to see that $Y$ is a closed sublattice of $C(\Omega)$.

It is proved in Theorem~3
in~\cite{Kakutani:41} that for a compact space $K$, every closed
sublattice $Y$ of $C(K)$ is of this form. That is
\begin{math}
  Y=\mathcal M_\perp,
\end{math}
where $\mathcal M$ consists of all the functionals of the form
$\delta_s-\alpha\delta_t$ with $s,t\in K$ and $\alpha\ge 0$ that
vanish on $Y$. Symbolically,
\begin{displaymath}
  Y=\Bigl(\bigl\{\delta_s-\alpha\delta_t\mid s,t\in K,\ \alpha\ge 0\bigr\}
  \cap Y^\perp\Bigr)_\perp.
\end{displaymath}
We are going to present a short proof of this and other
characterizations of closed sublattices $C(\Omega)$ where $\Omega$ is not
assumed to be compact.

% \begin{lemma}\label{plus-minus}
%   Let $Y$ be a sublattice of $C(\Omega)$.
%   The following three sets are equal:
%   \begin{eqnarray*}
%     Z_1&=&\Bigl(\bigl\{\alpha\delta_s+\beta\delta_t \mid s,t\in\Omega,\
%            \alpha,\beta\in\mathbb R\bigr\} \cap Y^\perp\Bigr)_\perp,\\
%     Z_2&=&\Bigl(\bigl\{\alpha\delta_s-\beta\delta_t \mid s,t\in\Omega,\
%            \alpha,\beta\ge 0\bigr\} \cap Y^\perp\Bigr)_\perp,\\
%     Z_3&=&\Bigl(\bigl\{\delta_s-\gamma\delta_t \mid s,t\in\Omega,\
%            \gamma\ge 0\bigr\} \cap Y^\perp\Bigr)_\perp.
%   \end{eqnarray*}
% \end{lemma}

% \begin{proof}
%   Each $Z_i$ is of the form $(A_i\cap Y^\perp)_\perp$. Since
%   $A_1\supseteq A_2\supseteq A_3$, we have
%   $Z_1\subseteq Z_2\subseteq Z_3$. Let $f\in Z_3$; we will show that
%   $\varphi\in Z_1$. Let $\varphi\in A_1$, i.e.,
%   $\varphi =\alpha\delta_s+\beta\delta_t$ for some
%   $\alpha,\beta\in\mathbb R$ and $\varphi\in Y^\perp$; we need to
%   show that $\varphi(f)=0$.

%   If either $\alpha$ or $\beta$ equals zero then $\phi\in A_3$ or
%   $-\varphi\in A_3$, hence $\varphi(f)=0$. Suppose now
%   $\alpha\ne 0\ne\beta$. WLOG, $\alpha=1$. If either $\delta_s$ or
%   $\delta_t$ is in $Y^\perp$ then the other one is also in $Y^\perp$
%   because $\varphi\in Y^\perp$; it follows that
%   $\delta_s,\delta_t\in A_3$ and, therefore, $\varphi(f)=0$. Suppose
%   now that $\delta_s,\delta_t\notin Y^\perp$. Then there exists
%   $g\in Y_+$ such that $g(s),g(t)>0$. It follows from
%   $g(s)+\beta g(s)=\varphi(g)=0$ that $\beta<0$, hence
%   $\varphi\in A_3$ and, therefore, $\varphi(f)=0$.
% \end{proof}

\begin{theorem}\label{CK-sublat-closure}
  Let $Y$ be a sublattice of $C(\Omega)$. Each of
  the following sets equals $\overline{Y}$.
  \begin{eqnarray*}
    Z_1&=&\Bigl\{f\in C(\Omega)\mid \forall\mbox{ finite }F\subseteq\Omega\
           \exists g\in Y,\mbox{ $f$ and $g$ agree on }F\Bigr\},\\
    Z_2&=&\Bigl\{f\in C(\Omega)\mid \forall s,t\in\Omega\
           \exists g\in Y,\ f(s)=g(s)\mbox{ and }f(t)=g(t)\Bigr\},\\
    Z_3&=&\Bigl(\bigl\{\delta_s-\gamma\delta_t \mid s,t\in\Omega,\
           \gamma\ge 0\bigr\} \cap Y^\perp\Bigr)_\perp,\\
    % Z_4&=&\Bigl(\bigl\{\alpha\delta_s-\beta\delta_t \mid s,t\in\Omega,\
    %        \alpha,\beta\ge 0\bigr\} \cap Y^\perp\Bigr)_\perp,\\
    % Z_5&=&\Bigl(\bigl\{\alpha\delta_s+\beta\delta_t \mid s,t\in\Omega,\
    %        \alpha,\beta\in\mathbb R\bigr\} \cap Y^\perp\Bigr)_\perp,\\
    Z_4&=&\Bigl(\Span\bigl\{\delta_s\mid s\in\Omega\bigr\}
           \cap Y^\perp\Bigr)_\perp,\\
    Z_5&=&\mbox{the intersection of all closed sublattices of } C(\Omega)\\
      &&\qquad\mbox{of codimension at most 1 that contain $Y$}.
  \end{eqnarray*}
\end{theorem}

\begin{proof}
  To show that $\overline{Y}\subseteq Z_1$, let $f\in\overline{Y}$ and
  $F=\{t_1,\dots,t_n\}$ in $\Omega$. Define
  $T\colon C(\Omega)\to\mathbb R^n$ via
  $Tg=\bigl(g(t_1),\dots,g(t_n)\bigr)$. Since $T$ is continuous,
  we have $TY\subseteq T\overline{Y}\subseteq\overline{TY}=TY$, hence
  $T\overline{Y}=TY$. It follows that there exists $g\in Y$ with
  $Tf=Tg$. Thus, $f\in Z_1$ and, therefore,
  $\overline{Y}\subseteq Z_1$.

  We trivially have $Z_1\subseteq Z_2$. We will now show that
  $Z_2\subseteq\overline{Y}$. Let $f\in Z_2$, let $\varepsilon>0$ and
  $K\subseteq\Omega$ be compact. It suffices to find $g\in Y$ such
  that $\bigabs{f(t)-g(t)}<\varepsilon$ for all $t\in K$.
  For every
  $s$ and $t$ in $K$ find $g_{s,t}\in Y$ so that $g_{s,t}(s)=f(s)$ and
  $g_{s,t}(t)=f(t)$. Put
  \begin{eqnarray*}
    U_{s,t} & = &
    \bigl\{w\in\Omega\mid g_{s,t}(w)<f(w)+\varepsilon\bigr\},
    \mbox{ and }\\
    V_{s,t} & = &
    \bigl\{w\in\Omega\mid g_{s,t}(w)>f(w)-\varepsilon\bigr\}.
  \end{eqnarray*}
  Clearly, each of these two sets is open and contains both $s$ and
  $t$. It follows that for each $t\in K$ the collection of sets
  $\{U_{s,t}\mid s\in K\}$ is an open cover of $K$. Hence, there is a
  finite subcover, $K=\bigcup_{i=1}^nU_{s_i,t}$. Put
  $g_t=\bigwedge_{i=1}^ng_{s_i,t}$. Then $g_t\in Y$, $g_t(t)=f(t)$,
  and $g_t\le f+\varepsilon\one$ on $K$. Put $V_t=\bigcap_{i=1}^nV_{s_i,t}$.
  Then $V_t$ is open, $t\in V_t$, and $g_t(w)\ge f(w)-\varepsilon$ for
  all $w\in V_t$. Again, the collection $\{V_t\mid t\in K\}$ is an
  open cover of $K$, hence there is a finite subcover
  $K=\bigcup_{j=1}^mV_{t_j}$. Put $g=\bigvee_{j=1}^mg_{t_j}$, then
  $g\in Y$ and $f-\varepsilon\one\le g\le f+\varepsilon\one$ on $K$.
  It follows that $Z_2\subseteq\overline{Y}$, so that
  $Z_1=Z_2=\overline{Y}$.

  We show next that $Z_3\subseteq Z_2$.  Suppose that $f\in Z_3$; fix
  $s\ne t$ in $\Omega$. We need to find $h\in Y$ such that $f(s)=h(s)$
  and $f(t)=h(t)$. Put $T\colon Y\to\mathbb R^2$ via
  $Th=\bigl(h(s),h(t)\bigr)$. If $\Range T$ is all of $\mathbb R^2$
  then we can find $h\in Y$ such that
  $\bigl(h(s),h(t)\bigr)=Th=\bigl(f(s),f(t)\bigr)$. If
  $\Range T=\{0\}$ then both $\delta_s$ and $\delta_t$ vanish on $Y$; it
  follows that they vanish on $f$ because $f\in Z_3$, and we can take
  $h=0$. Finally, suppose that $T$ has rank 1. Then, possibly after
  interchanging $s$ and $t$, we find $\alpha\in\mathbb R$ such that
  $g(s)=\alpha g(t)$ for all $g\in Y$. It also follows that
  $g_0(t)\ne 0$ for some $g_0\in Y$. Since $\abs{g_0}\in Y$, we have
  $\abs{g_0(s)}=\alpha\abs{g_0(t)}$, which yields $\alpha\ge 0$. Let
  $\mu:=\delta_s-\alpha\delta_t$. Since $f\in Z_3$, we have
  $f\in\ker\mu$. We take $h$ to be a scalar multiple of $g_0$ so
  that $h(t)=f(t)$. It now follows that
  $h(s)=\alpha h(t)=\alpha f(t)=f(s)$. Hence, $f\in Z_2$. We now
  conclude from
  $\overline{Y}\subseteq Z_3\subseteq Z_2=\overline{Y}$
  that $Z_3=\overline{Y}$.

  It is easy to see that each $Z_4$ and $Z_5$ contains $\overline{Y}$
  and is contained in $Z_3$. Since we already know that
  $Z_3=\overline{Y}$, this completes the proof.
\end{proof}

\begin{corollary}\label{CK-sublat-closed}
  Closed sublattices of $C(\Omega)$ are the sets of the form $\mathcal M_\perp$,
  where $\mathcal M$ is a collection of functionals of the form
  $\delta_s-\alpha\delta_t$ where $s,t\in\Omega$ and
  $\alpha\ge 0$.
\end{corollary}

\begin{corollary}
  Every closed sublattice of $C(\Omega)$ of codimension one is either
  of the form $\bigl\{f\in C(\Omega)\mid f(s)=0\bigr\}$ for some
  $s\in \Omega$ or of the form
  $\bigl\{f\in C(\Omega)\mid f(s)=\alpha f(t)\bigr\}$ for some
  $s\ne t$ in $\Omega$ and $\alpha>0$.
\end{corollary}

\begin{remark}\label{equiv-classes}
  Let $Y$ be a sublattice of $C(\Omega)$. By
  Theorem~\ref{CK-sublat-closure}, $\overline{Y}=\mathcal M_\perp$
  where
  \begin{math}
    \mathcal M=\bigl\{\delta_s-\alpha\delta_t \mid s,t\in\Omega,\
           \alpha\ge 0\bigr\} \cap Y^\perp.
  \end{math}
  Let $\ker Y=\bigl\{t\in\Omega\mid\delta_t\in\mathcal M\bigr\}$. That
  is, $t\in\ker Y$ if and only if $f(t)=0$ for all $f\in Y$. It follows from
  $\ker Y=\bigcap_{f\in Y}\ker f$ that $\ker Y$ is closed. For any
  $s,t\in\Omega\setminus\ker Y$, we set $s\sim t$ if there exists
  $\alpha\ne 0$ such that $f(s)=\alpha f(t)$ for all $f\in Y$.  Note
  that $\alpha$ has to be positive. Indeed, since $s,t\notin\ker Y$,
  there exist $f,g\in Y$ with $f(s)\ne 0$ and $g(t)\ne 0$. Put
  $h=\abs{f}\vee\abs{g}$; then $h\in Y$. It follows from
  $h(s)=\alpha h(t)$ and $h(s),h(t)>0$ that $\alpha>0$. It also
  follows that $\alpha$ is uniquely determined by $s$ and $t$. One can
  easily verify that $s\not\sim t$ if and only if there exists $f\in Y_+$ such
  that $f(s)=1$ and $f(t)=0$. This yields that
  $s\sim t$ is equivalent to the following condition: $f(s)=0$ if and only if
  $f(t)=0$ for all $f\in Y$.

  It is clear that $s\sim t$ if and only if $\delta_s$ and $\delta_t$ are
  proportional on $Y$ and, therefore, on $\overline{Y}$. This
  is an equivalence relation on $\Omega\setminus\ker Y$; hence yields
  a partition of $\Omega\setminus\ker Y$ into equivalence classes
  $\{A_\gamma\}$. We will call these sets the \term{clans} of $Y$.
  Of course, of primary interests are the clans
  consisting of more than one point.

  Let $A_\gamma$ be a clan. Fix $e_\gamma\in Y_+$ which does not
  vanish at some point (and, therefore, at every point) of
  $A_\gamma$. Furthermore, let $f\in Y$ be arbitrary and fix $t\in
  A_\gamma$. For every $s\in A_\gamma$, it follows from $s\sim t$ that
  $\frac{f(s)}{e_\gamma(s)}=\frac{f(t)}{e_\gamma(t)}$. Therefore,
  every function in $Y$ is proportional to $e_\gamma$ on
  $A_\gamma$. It now follows that for $f\in C(\Omega)$, we have
  $f\in\overline{Y}$ if and only if $f$ vanishes on $\ker Y$ and is proportional
  to $e_\gamma$ on $A_\gamma$ for every $\gamma$.
\end{remark}

\begin{corollary}\label{CK-sublat-proport}
  Let $Y$ be a sublattice of $C(\Omega)$. Then there exist disjoint
  subsets $A_0$ and $(A_\gamma)_{\gamma\in\Gamma}$ and functions
  $(e_\gamma)_{\gamma\in\Gamma}$ in $Y$ such that $f\in \overline{Y}$
  if and only if $f$ vanishes on $A_0$ and is proportional to $e_\gamma$ on
  $A_\gamma$ for every $\gamma\in\Gamma$.
\end{corollary}

It is easy to see that the converse is also true: every set of this
form is a closed sublattice of $C(\Omega)$.
\medskip

Now consider the special case when $Y$ is an ideal in $C(\Omega)$. It
is easy to see that if $\delta_s-\alpha\delta_t$ vanishes on $Y$ for
some $s\ne t$ and $\alpha>0$ then $\delta_s$ and $\delta_t$ vanish on
$Y$, so that $s,t\in\ker Y$ in the notation of
Remark~\ref{equiv-classes}. It is clear than that all the equivalence
classes in $\Omega\setminus\ker Y$ are singletons and $\mathcal M$
consists of evaluation functionals only.

\begin{corollary}\label{closed-ideal}
  Every closed ideal in $C(\Omega)$ is of the form
  \begin{math}
    J_F=\bigl\{f\in C(\Omega)\mid f_{|F}=0\bigr\}
  \end{math}
  for some closed set $F$. More generally,
  if $J$ is an ideal of $C(\Omega)$ then $\overline{J}=J_F$ where
  $F=\ker J$.
\end{corollary}

\begin{corollary}\label{CK-sublat-ideal}
  Every closed sublattice $Y$ of $C(\Omega)$ of codimension $n$ contains a
  closed ideal $J$ of $C(\Omega)$ of codimension at most $2n$.
\end{corollary}

\begin{proof}
  By Theorem~\ref{CK-sublat-closure}, $Y=\mathcal
  M_\perp=(\Span\mathcal M)_\perp$ where
  \begin{displaymath}
    \mathcal M=\bigl\{\delta_s-\gamma\delta_t \mid s,t\in\Omega,\
    \gamma\ge 0\bigr\} \cap Y^\perp.
  \end{displaymath}
  It follows that $Y=\bigcap_{i=1}^n\ker\mu_i$ for some linearly
  independent $\mu_1,\dots,\mu_n\in\mathcal M$, so each $\mu_i$ is the
  form $\delta_s$ or $\delta_s-\alpha\delta_t$ for some $s,t\in\Omega$
  and $\alpha\ge 0$. Put $J$ to be the intersection of the kernels of
  all these $\delta_s$'s and $\delta_t$'s.
\end{proof}

\begin{example}\label{one-clan}
  Let $Y$ be the one-dimensional sublattice of $C[0,1]$ consisting of
  all scalar multiples of $f(t)=t$. Then $\ker Y=\{0\}$ and the only
  clan of $Y$ is $(0,1]$.
\end{example}

Again, let $Y$ be a sublattice of $C(\Omega)$. It is clear that
$\ker Y$ is closed. Example~\ref{one-clan} shows that clans need not
be closed. However, we have the following:

\begin{proposition}
  Let $Y$ be a sublattice of $C(\Omega)$.
  Clans of $Y$ are closed in $\Omega\setminus\ker Y$.
\end{proposition}

\begin{proof}
  Fix a clan $A_\gamma$; let $e_\gamma$ be as in
  Remark~\ref{equiv-classes}. We know that for every $f\in Y$ there
  exists $\lambda_f\ge 0$ such that $f$ agrees with
  $\lambda_fe_\gamma$ on $A_\gamma$. Let $F$ be the set of all $t$
  such that $f(t)=\lambda_fe_\gamma(t)$ for all $f\in Y$.
  Clearly, $F$ is closed and contains $A_\gamma\cup\ker Y$. It
  suffices to show that $F\subseteq A_\gamma\cup\ker Y$ as this would
  imply $A_\gamma=F\cap(\Omega\setminus\ker Y)$.  Let
  $s\not\in A_\gamma\cup\ker Y$ but $s\in F$. Then there exists
  $t\in A_\gamma$ such that $t\not\sim s$. It follows that there
  exists $f\in Y$ such that $f(s)\ne 0$ but $f(t)=0$. But the latter
  implies that $\lambda_f=0$. Since $s\in F$, it follows that
  $f(s)=0$, which is a contradiction.
\end{proof}

There are well known characterizations of bands and projection bands
in $C(\Omega)$ spaces; see, e.g., Corollary~2.1.10
in~\cite{Meyer-Nieberg:91} for the case when $\Omega$ is compact and
Corollary~4.5 in~\cite{Kandic:19} for the non-compact case.

\section{Closed sublattices of Banach sequence spaces}

By a Banach sequence space we mean a Banach lattice whose order is
given by a Schauder basis. It is clear that the basis has to be
1-unconditional. One can view a Banach sequence space as a sublattice
of $\mathbb R^{\mathbb N}$, which may also be identified with
$C(\mathbb N)$, where $\mathbb N$ is equipped with the discrete
topology. When talking about a sequence of vectors in a Banach
sequence space, it is sometimes convenient to use subscripts and
superscripts. For example, $x^{(n)}_i$ stands for the $i$-th
coordinate of the $n$-th term of the sequence $(x^{(n)})$. The
following result is folklore; we include the proof for the convenience
of the reader.

\begin{proposition}\label{basis-ocn}\label{bss-ocn}
  Every Banach sequence space is order continuous.
\end{proposition}

\begin{proof}
  Let $X$ be a Banach sequence space; we write $P_n$ for the $n$-th
  basis projection, that is $(P_nx)_i$ equals $x_i$ is $i\le n$ and
  zero otherwise. Let $(x^{(\alpha)})$ be a net in $X$; suppose that
  $x^{(\alpha)}\downarrow 0$; we will prove that
  $\norm{x^{(\alpha)}}\to 0$. Let $\varepsilon>0$. Fix any index
  $\alpha_0$. There exists $n_0$ such that
  $\norm{x^{(\alpha_0)}-P_{n_0}x^{(\alpha_0)}}<\varepsilon$. Since
  $x^{(\alpha)}\downarrow 0$, the net converges to zero coordinate-wise. It
  follows that $P_{n_0}x^{(\alpha)}\to 0$ as
  $\alpha\to\infty$. Therefore, there exists $\alpha_1\ge\alpha_0$
  such that $\norm{P_{n_0}x^{(\alpha_1)}}<\varepsilon$. We conclude
  that
    \begin{displaymath}
      \norm{x^{(\alpha)}}\le\norm{P_{n_0}x^{(\alpha)}}+\norm{(I-P_{n_0})x^{(\alpha)}}
      \le\norm{P_{n_0}x^{(\alpha_1)}}+\norm{(I-P_{n_0})x^{(\alpha_0)}}
      <2\varepsilon.
    \end{displaymath}
    whenever $\alpha\ge\alpha_1$.
\end{proof}

The following theorem extends Theorem~5.2 of~\cite{Radjavi:08}. We
first provide a direct proof of it; we will later show that it may be
easily deduced from a more general statement.

\begin{theorem}\label{sublat-discr}
  Let $X$ be a Banach sequence space. Closed sublattices of $X$ are
  exactly the closed spans of (finite or infinite) disjoint positive
  sequences.
\end{theorem}

\begin{proof}
  Let $Y$ be a closed sublattice of $X$.  As in
  Remark~\ref{equiv-classes}, we consider
  $\ker Y=\{i\in\mathbb N\mid\forall y\in Y\ y_i=0\}$; for
  $i,j\notin\ker Y$, we write $i\sim j$ if there exists $\alpha>0$
  such that $y_i=\alpha y_j$ for every $y\in Y$. This is an equivalence
  relation on $\mathbb N\setminus\ker Y$; the equivalence classes are
  the clans of $Y$. Again, as in Remark~\ref{equiv-classes}, if
  $i\not\sim j$ then there exists $y\in Y_+$ such that $y_i=1$ and
  $y_j=0$.

  Fix a clan $A$ and any $n\in A$. Let $P_A$ be the natural basis
  projection onto $A$, that is, $P_Ax$ is the sequence that agrees
  with $x$ on $A$ and vanishes outside of $A$. It is easy to see that
  $P_Ax$ and $P_Ay$ are proportional for any $x,y\in Y$, hence
  $P_A(Y)$ is one-dimensional. Enumerating
  $\mathbb N\setminus(\ker Y\cup A)$, we produce a sequence
  $(y^{(m)})$ in $Y_+$ such that $y^{(m)}_n=1$ for all $m$, and for
  every $j\notin A$ there exists $m$ with $y^{(m)}_j=0$. Replacing
  $y^{(m)}$ with $y^{(1)}\wedge\dots \wedge y^{(m)}$, we may assume
  that $y^{(m)}\downarrow$. Since $X$ is order complete and
  $y^{(m)}\ge 0$, we have $y^{(m)}\downarrow x$ for some $x\in
  X_+$. Since $X$ is order continuous, $y^{(m)}$ converges to $x$ in
  norm, hence $x\in Y$. Also, it follows from $y^{(m)}\downarrow x$
  that $x_n=1$ and $x_i=0$ for all $i\notin A$. Clearly, $x_k\ne 0$
  for all $k\in A$.  This yields that the support of $x$ is $A$. It
  follows that $P_A(Y)=\Span\{x\}$.

  Let $A_1,A_2,\dots$ be the clans of $Y$. For each $n$, fix
  $x^{(n)}\in Y_+$ such that $P_{A_n}(Y)=\Span x^{(n)}$. Clearly,
  $(x^{(n)})$ is a disjoint sequence in $Y$. Take any $z\in Y$, it
  suffices to show that $z\in\Span\{x^{(n)}\}$. Without
  loss of generality, $z\ge 0$ as, otherwise, we can consider $z^+$
  and $z^-$. For each $m$, put $z^{(m)}=P_{A_1}z+\dots+P_{A_m}z$. Then
  $z^{(m)}\in\Span\{x^{(n)}\}$. It follows from $z^{(m)}\uparrow z$
  that that $z^{(m)}$ converge to $z$ in norm; hence
  $z\in\Span\{x^{(n)}\}$.
\end{proof}

The following is an easy corollary of
Theorem~\ref{sublat-discr}; it may also be deduced from
Corollary~\ref{CK-sublat-proport} by taking $\Omega=\{1,\dots,n\}$.

\begin{corollary}\label{sublat-Rn}
  A subspace of $\mathbb R^n$ is a sublattice if and only if it can be
  written as a span of disjoint positive vectors.
\end{corollary}

\begin{remark}\label{sublat-discr-gen}
  Consider the following generalization of
  Theorem~\ref{sublat-discr}. Let $\Omega$ be an arbitrary set, $X$ an
  order ideal in $\mathbb R^\Omega$ equipped with a complete lattice
  norm, such that finitely supported functions are dense in $X$ and
  norm convergence in $X$ implies point-wise convergence. We claim
  that every norm closed sublattice of $X$ is the closed span of a
  disjoint family of positive vectors.
\end{remark}

This claim may be proved by adapting the proof of
Theorem~\ref{sublat-discr} above. Instead, we will provide a different
argument that proves both Theorem~\ref{sublat-discr} and
Remark~\ref{sublat-discr-gen} by reducing them to our characterization
of closed sublattices in $C(\Omega)$ is
Section~\ref{CK-cl-sublat}. Recall that a net $(x_\alpha)$ in a vector
lattice $X$ \term{uo-converges} to $x$ if $\abs{x_\alpha-x}\wedge u$
converges to zero in order for every $u\ge 0$. A sublattice is order
closed if and only if it is uo-closed. Furthermore, let $(x_\alpha)$
be a net in a regular sublattice $Y$ of $X$, then $x_\alpha\goesuo 0$
in $X$ if and only if $x_\alpha\goesuo 0$ in $Y$.  We refer the reader
to~\cite{Gao:17} for a review of order convergence, uo-convergence,
and regular sublattices.

\begin{proposition}\label{ocont-COmega}
  Let $X$ be an order continuous Banach lattice, continuously embedded
  as a sublattice into $C(\Omega)$ for some locally compact $\Omega$;
  let $Y$ be a sublattice of $X$. If $Y$ is norm closed in $X$
  then $Y$ is closed in $X$ with respect to the compact-open topology
  of $C(\Omega)$ restricted to $X$. It follows that
  $Y=\overline{Y}^{C(\Omega)}\cap X$.
\end{proposition}

\begin{proof}
  We claim that $X$ is regular in $C(\Omega)$. Indeed, suppose that
  $x_\alpha\downarrow 0$ in $X$. It follows that $x_\alpha\to 0$ in
  norm of $X$ and, therefore, in the compact-open topology of
  $C(\Omega)$. Since $(x_\alpha)$ is monotone, we have
  $x_\alpha\downarrow 0$ in $C(\Omega)$ by Theorem~2.21c
  of~\cite{Aliprantis:03}. This proves the claim.

  Since $Y$ is norm closed in $X$ and $X$ is order continuous, $Y$ is
  order closed in $X$, and, therefore, uo-closed in $X$. It follows
  easily from Theorem~3.2 in~\cite{Bilokopytov:22} that compact-open
  convergence on $C(\Omega)$ is stronger than uo-convergence in
  $C(\Omega)$. The same holds true for their restrictions to $X$
  because compact-open convergence is topological and $X$ is
  regular. It follows that $Y$ is closed in the restriction of the
  compact-open topology to~$X$.
\end{proof}

We can now deduce the claim in Remark~\ref{sublat-discr-gen}
from Proposition~\ref{ocont-COmega} and Corollary~\ref{CK-sublat-proport}:

\begin{proof}[Proof of Remark~\ref{sublat-discr-gen}]
  As in Proposition~\ref{bss-ocn}, we can show that $X$ is order
  continuous. We equip $\Omega$ with the discrete topology, so that
  $\mathbb R^\Omega=C(\Omega).$ By Proposition~\ref{ocont-COmega}, for
  every closed sublattice $Y$ of $X$, we have $Y=\overline{Y}\cap X$ ,
  where $\overline{Y}$ is the closure of $Y$ in $C(\Omega)$, hence is
  a closed sublattice of $C(\Omega)$. We now apply
  Corollary~\ref{CK-sublat-proport} to $\overline{Y}$; let $A_0$,
  $A_\gamma$ and $e_\gamma$ be as in the corollary.  Note that
  $e_\gamma\cdot\one_{A_\gamma}$ is in $\overline{Y}$ by
  Corollary~\ref{CK-sublat-proport}. Since $X$ is an ideal in
  $C(\Omega)$, we have $e_\gamma\cdot\one_{A_\gamma}$ is in $X$ and,
  therefore, in $Y$ because $Y=\overline{Y}\cap X$.  Hence, replacing
  each $e_\gamma$ with $e_\gamma\cdot\one_{A_\gamma}$, we may assume
  that the collection $(e_\gamma)$ is disjoint.  It follows from
  Corollary~\ref{CK-sublat-proport} that $Y$ is the closed linear span
  of this collection.
\end{proof}

\section{Uniformly closed subspaces of finite codimension}

Throughout the rest of the paper, $X$ stands for an Archimedean vector
lattice.  For $e\in X_+$, we write $I_e$ for the principal ideal of
$e$. For $x\in X$, we put
\begin{displaymath}
  \norm{x}_e=\inf\bigl\{\lambda>0\mid\abs{x}\le\lambda e\bigr\}.
\end{displaymath}
This defines a lattice norm on $I_e$. By Kakutani-Krein representation
theorem, $\bigl(I_e,\norm{\cdot}_e\bigr)$ is lattice isometric to a
dense sublattice of $C(K)$ for some compact Hausdorff space~$K$.

For a net $(x_\alpha)$ in $X$, we say that it converges
\term{relatively uniformly} or just \term{uniformly} to $x\in X$ and
write $x_\alpha\goesu x$ if $\norm{x_\alpha-x}_e\to 0$ for some
$e\in X_+$; this implies, in particular, that $x_\alpha-x\in I_e$ for
all sufficiently large $\alpha$. A subset $A$ of $X$ is said to be
\term{uniformly closed} if it is closed with respect to uniform
convergence.  Equivalently, $A\cap I_e$ is closed in
$\bigl(I_e,\norm{\cdot}_e\bigr)$ for every $e\in X_+$. Actually, it
suffices to verify this condition for every $e$ in a majorizing
sublattice of $X$ because $e\le u$ implies
$\norm{\cdot}_u\le\norm{\cdot}_e$. In a Banach lattice, every
uniformly convergent sequence is norm convergent and every norm
convergent sequence has a uniformly convergent subsequence (see, e.g.,
Lemma~1.1 of~\cite{Taylor:20}). It follows that a subset of a Banach
lattice is uniformly closed if and only if it is norm closed.

Recall that if $J$ is an ideal in $X$ then $X/J$ is a vector lattice
and the quotient map is a lattice homomorphism; $X/J$ is Archimedean
if and only if $J$ is uniformly closed by Theorem~60.2 in~\cite{Luxemburg:71}.

A vector lattice is \term{uniformly
  complete} if for every $e\in X_+$ the ideal $I_e$ is complete in
$\norm{\cdot}_e$ and, as a consequence,
$\bigl(I_e,\norm{\cdot}_e\bigr)$ is lattice isometric to $C(K)$ for
some compact Hausdorff space $K$.

If $K$ is a compact Hausdorff space then (relative) uniform
convergence in $C(K)$ agrees with compact-open convergence, which is,
in this case, the same as the norm convergence, i.e., the
convergence in $\norm{\cdot}_{\one}$ norm. This is no longer true when
$\Omega$ is just a Tychonoff space, which creates an unfortunate clash
of terminologies. In this paper, uniform convergence in $C(\Omega)$
will always be interpreted as the relative uniform convergence rather
than convergence in $\norm{\cdot}_{\one}$. It is easy to see that
uniform convergence in $C(\Omega)$ implies convergence in compact-open
topology, so that sets that are closed in compact-open topology are
uniformly closed.

The converses to both implications in the preceding sentence are
false, as the following two examples show. In the first example, we
construct a net that converges in the compact-open topology but not
uniformly. Identify $\mathbb R^{\mathbb N}$ with $C(\mathbb N)$ and
consider the double sequence $x_{n,m}=(0,\dots,0,n,n,\dots)$, with $m$
zeros at the head. It is easy to see that $(x_{n,m})$ converges to
zero pointwise, hence in the compact-open topology. However, every
tail of this net fails to be order bounded, hence it cannot converge
uniformly. In the second example, we construct a set that is uniformly
closed but fails to be closed in the compact-open topology. Let
$\Omega$ be an uncountable set equipped with the discrete topology,
and let $J$ be the order ideal in $\mathbb R^\Omega=C(\Omega)$
consisting of all functions with countable support. It is easy to see
that $J$ is uniformly closed, yet it is dense in $C(\Omega)$ in the
compact-open topology.

We note, however, that in the special case when $\Omega$ is locally
compact and $\sigma$-compact, i.e., can be written as a union of
countably many compact sets, uniformly closed sets are compact-open
closed. Indeed, it suffices to show that every net that converges in
compact-open topology contains a sequence that converges
uniformly. Since compact-open topology on $C(\Omega)$ is complete and
metrizable by, e.g.,~\cite[pp.~62-64]{Beckenstein:77}, the proof of
this fact is similar to the classical proof that every norm convergent
sequence in a Banach lattice has a uniformly convergent subsequence;
see, e.g., \cite[Lemma~1.1]{Taylor:20} \medskip

It was shown in \cite[Theorem~5.1]{Taylor:20} that an
operator between vector lattices is order bounded if and only if it is
\term{uniformly continuous}, i.e., maps uniformly convergent nets to
uniformly convergent nets.

\begin{proposition}\label{obdd-ker-u-closed}
  A linear functional $\varphi$ on $X$ is order bounded if and only if
  $\ker\varphi$ is uniformly closed.
\end{proposition}

\begin{proof}
  If $\varphi$ is order bounded, it is uniformly continuous; it
  follows that $\ker\varphi$ is uniformly closed. Conversely, suppose
  that $\ker\varphi$ is uniformly closed. We need to show that
  $\varphi$ is uniformly continuous. It suffices to show that the
  restriction $\varphi_{|I_e}$ of $\varphi$ to $I$ is
  $\norm{\cdot}_e$-continuous for every $e\in X_+$. This is true
  because $\ker\varphi_{|I_e}=\ker\varphi\cap I_e$ is closed in
  $\bigl(I_e,\norm{\cdot}_e\bigr)$.
%  which is
%  equivalent to $\ker\varphi_{|I_e}$ being closed in
%  $\bigl(I_e,\norm{\cdot}_e\bigr)$; the latter is straightforward.
\end{proof}

Since every subspace of codimension one is the kernel of a linear
functional, we get the following:

\begin{corollary}\label{u-subsp-codim-1}
  Uniformly closed subspaces of $X$ of codimension 1 are exactly the
  kernels of order bounded functionals.
\end{corollary}

\begin{example}
  Let $X=c_{00}$, the space of all eventually zero sequences. Every
  subspace $Y$ of $c_{00}$ is uniformly closed because for every
  $e\in X_+$, the principal ideal $I_e$ is finite-dimensional, hence
  $Y\cap I_e$ is closed in $\bigl(I_e,\norm{\cdot}_e\bigr)$. It
  follows from Proposition~\ref{obdd-ker-u-closed} that every linear
  functional on $c_{00}$ is order bounded.
\end{example}

\begin{lemma}
  Let $Y$ and $Z$ be subspaces of $X$ such that $Y$ is uniformly
  closed and $Z$ is finite dimensional. Then $Y+Z$ is uniformly
  closed.
\end{lemma}

\begin{proof}
  Suppose that $x_\alpha\goesu x$ for some $(x_\alpha)$ in $Y+Z$. Find
  $e\in X_+$ with $\norm{x_\alpha-x}_e\to 0$. Since $Z$ is
  finite-dimensional, we can find $u\ge e$ such that $Z\subseteq
  I_u$. Then $(Y+Z)\cap I_u=(Y\cap I_u)+Z$. Since $Y\cap I_u$ is
  closed in $I_u$, so is $(Y+Z)\cap I_u$. It follows from
  $\norm{x_\alpha-x}_u\to 0$ that $x\in Y+Z$.
\end{proof}

\begin{corollary}\label{nested}
  Let $Y$ be a uniformly closed subspace of $X$ of finite
  codimension. If $Y\subseteq Z\subseteq X$ for some subspace $Z$ then
  $Z$ is uniformly closed.
\end{corollary}

\begin{corollary}\label{ucmpl-cl-n}
  Every uniformly closed subspace of $X$ of finite codimension is an
  intersection of uniformly closed subspaces of codimention~1.
\end{corollary}

\begin{question}
  Under what conditions on $X$ can we say that \emph{every} uniformly
  closed subspace is the intersection of uniformly closed subspaces of
  codimension~1? This is not true in general as $X^\sim$ may be
  trivial, and then $X$ has no uniformly closed subspaces of
  codimension~1. However, this is true when $X$ is a Banach lattice
  because in this case uniformly closed and norm closed sets agree.
\end{question}

Let $Y$ be a subspace of $X$ of co-dimension $n$. By
Corollaries~\ref{u-subsp-codim-1} and~\ref{ucmpl-cl-n}, $Y$ is
uniformly closed iff $Y=\bigcap_{i=1}^n\ker\varphi_i$ for some order
bounded functionals $\varphi_1,\dots,\varphi_n$. This is equivalent to
$Y$ being the kernel of an order bounded linear operator
$T\colon X\to\mathbb R^n$, --- just take $(Tx)_i=\varphi_i(x)$. This
motivates the following question:

\begin{question}\label{q:ussubs}
  Is every uniformly closed subspace the kernel of an
  order bounded operator?
\end{question}

In Proposition~\ref{obdd-ker-u-closed}, we describe the kernels of
order bounded functionals. The following result characterizes kernels
of positive functionals. Recall that a subspace $Y$ of $X$ is
\term{full} if $x,y\in Y$ and $x\le y$ imply $[x,y]\subseteq
Y$. Clearly, every ideal is a full subspace. The converse is false:
the straight line $y=-x$ in $\mathbb R^2$ is full but not an ideal.

\begin{lemma}
  A subspace $Y$ is full if and only if the set
  \begin{math}
    \bigl\{x\in X\mid\abs{x}\in Y\bigr\}
  \end{math}
  is an ideal.
\end{lemma}

\begin{proof}
  Suppose that $Y$ is full; we will show that the set
  $J:=\bigl\{x\in X\mid\abs{x}\in Y\bigr\}$ is an ideal. It is clear
  that $J$ is closed under scalar multiplication. To show that it is
  closed under addition, let $x,y\in J$, then $\abs{x}+\abs{y}\in Y$;
  it now follows from $0\le\abs{x+y}\le\abs{x}+\abs{y}$ that
  $\abs{x+y}\in Y$ and, therefore, $x+y\in J$. Finally, suppose that
  $y\in J$ and $\abs{x}\le\abs{y}$; it follows from
  $0\le\abs{x}\le\abs{y}$ that $\abs{x}\in Y$ and, therefore, $x\in
  J$.

  Conversely, suppose that $J$ is an ideal. Suppose $x\in[y_1,y_2]$
  where $y_1,y_2\in Y$. It follows from $0\le x-y_1\le y_2-y_1$ and
  $y_2-y_1\in J$ that $x-y_1\in J$, hence $x-y_1\in Y$ and, therefore,
  $x\in Y$.
\end{proof}

\begin{proposition}\label{ker-full}
  A one-codimensional subspace of $X$ is full if and only if it is the kernel of
  a positive functional.
\end{proposition}

\begin{proof}
  Let $\varphi$ be a positive functional on $X$,
  $y_1,y_2\in\ker\varphi$, and $y_1\le x\le y_2$. Then
  $0=\varphi(y_1)\le\varphi(x)\le\varphi(y_2)=0$ implies
  $x\in\ker\varphi$.

  For the converse, suppose that $\ker\varphi$ is full but $\varphi$
  is neither positive nor negative. Then there exist $x,y\in X_+$ such
  that $\varphi(x)>0$ and $\varphi(y)<0$. Take $\alpha,\beta>0$
  such that $\alpha+\beta=1$ and
  $0=\alpha\varphi(x)+\beta\varphi(y)=\varphi\bigl(\alpha x+\beta
  y)$. It now follows from $0\le\alpha x\le\alpha x+\beta
  y\in\ker\varphi$ that $\alpha x\in\ker\varphi$, which is a contradiction.
\end{proof}

The following two results are known; we include them for completeness.

\begin{proposition}\label{ker-ideal}
  A subspace $J$ of co-dimension 1 is an ideal if and only if $J$ is the kernel
  of a real-valued lattice homomorphism.
\end{proposition}

\begin{proof}
  If $J$ is an ideal then $J$ is the kernel of the quotient map $X\to
  X/J$. The converse is trivial.
\end{proof}

For a positive functional $\varphi$, we write $N_\varphi$ for its
\term{null ideal}:
$\bigl\{x\in X\mid\varphi\bigl(\abs{x}\bigr)=0\bigr\}$.

\begin{proposition}
  An ideal $J$ in $X$ is a null ideal of some positive functional if and only if
  $X/J$ admits a strictly positive functional.
\end{proposition}

% \begin{proof}
%   Suppose that $J=N_\varphi$ for some $\varphi>0$. It follows from
%   $N_\varphi\subseteq\ker\varphi$ that $\varphi$ induces a functional
%   on $X/N_{\varphi}$ via $\tilde\varphi(x+N_\varphi)=\varphi(x)$. If
%   $0<x\not\in N_{\varphi}$ then $\tilde\varphi(x+N_\varphi)>0$; it
%   follows that $\tilde\varphi$ is strictly positive.

%   To prove the converse, let $\psi$ be a strictly positive functional
%   on $X/J$. Put $\varphi=\psi\circ Q$, where $Q\colon X\to X/J$ is the
%   canonical quotient map. In particular, $Q$ is a lattice
%   homomorphism. Clearly, $\varphi\ge 0$. For $x\in X$, we have
%   \begin{displaymath}
%     \varphi\bigl(\abs{x}\bigr)=0
%     \quad\Leftrightarrow\quad
%     \psi\bigl(\abs{Qx}\bigr)=0
%     \quad\Leftrightarrow\quad
%     Qx=0
%     \quad\Leftrightarrow\quad
%     x\in J,
%   \end{displaymath}
%   hence $J=N_\varphi$.
% \end{proof}

\section{Uniformly closed sublattices of finite codimension}

The following result is a part of Theorem~3 in~\cite{Abramovich:90a}. As it is
important for our exposition, we include a short proof of it.

\begin{proposition}[\cite{Abramovich:90a}]\label{sublats-exist}
  For every $n\in\mathbb N$ with $n\le\dim X$, $X$ admits a sublattice
  of codimension $n$.
\end{proposition}

\begin{proof}
  Clearly, it suffices to prove the statement for $n=1$. By
  Theorem~33.4 in~\cite{Luxemburg:71} $X$ admits a proper ideal $J$
  which is prime, i.e., $x\wedge y\in J$ implies $x\in J$ or $y\in
  J$. By Zorn's Lemma, there exists a subspace $Y$ of $X$ of
  codimension 1 containing $J$. If $x\in Y$ then $x^+$ or $x^-$ is in
  $J$, hence both $x^+$ and $x^-$ are in $Y$. Therefore, $Y$ is a
  sublattice.
\end{proof}

% Furthermore, it follows from Corollary~3 in~\cite{Abramovich:90a} that
% every Banach lattice admits a (discontinuous) functional whose kernel
% is a sublattice that is not uniformly closed. The following
% proposition does not require a sublattice to be uniformly closed.

\begin{lemma}
  Let $Y$ be a sublattice of $\mathbb R^n$ of codimension~$m$. Then
  $Y$ contains at least $n-2m$ and at most $n-m$ of the standard unit
  vectors.
\end{lemma}

\begin{proof}
  The lower estimate follows from Corollary~\ref{sublat-Rn}; the upper
  estimate is obvious.
\end{proof}

The following result was proved in \cite{Abramovich:90b} in the
special case of sublattices of codimension~1.

\begin{proposition}\label{sublat-disj-2m}
  Let $Y$ be a sublattice of $X$ of codimension $m$. Then any disjoint
  set that does not meet $Y$ consists of at most $2m$ vectors.
\end{proposition}

\begin{proof}
  Suppose that $x_1,\dots,x_n$ are disjoint vectors in $X\setminus Y$
  for some $n>2m$. Assume first that $x_1,\dots,x_n>0$. Let
  $Z=\Span\{x_1,\dots,x_n\}$. Clearly, there exists a lattice
  isomorphism $T$ from $Z$ onto $\mathbb R^n$ such that $Tx_i=e_i$ as
  $i=1,\dots,n$. Note that $Y\cap Z$ is a sublattice of $Z$ of
  codimension at most $m$. It follows from the preceding lemma that
  $Y\cap Z$ contains some (at least $n-2m$) of the $x_i$'s, which
  contradicts the assumptions. This proves the theorem for positive
  vectors.

  Suppose now that $x_1,\dots,x_n$ are arbitrary disjoint vectors in
  $X\setminus Y$ for some $n>2m$. For each $i$, either $x_i^+$ or
  $x_i^-$ (or both) is not in $Y$ (in particular, it is
  non-zero). This yields a collection of $n$ disjoint positive vectors
  in $X\setminus Y$, which is impossible by the first part of the
  proof.
\end{proof}

\bigskip

We now proceed to uniformly closed sublattices. Recall that $X$ is
assumed to be an Archimedean vector lattice (the Archimedean
property is not necessary for the three preceding results). We proved in
Corollary~\ref{u-subsp-codim-1} that uniformly closed subspaces of $X$
of codimension 1 are the kernels of order bounded functionals. Also,
we observed that every closed sublattice $Y$ of $C(\Omega)$ of
codimension 1 is the kernel of $\delta_t-\alpha\delta_s$ for some
$s\ne t$ in $\Omega$ and $\alpha\ge 0$; hence, $Y$ is the kernel of
the difference of two lattice homomorphisms in $C(\Omega)^*$. In
Theorem~2b of~\cite{Abramovich:90b}, the same result was established
for vector latices of simple functions. We are now going to extend
this to general vector lattices.

Recall that a positive operator $T$ between vector lattices is a
lattice homomorphism if and only if $x\wedge y=0$ implies $Tx\wedge Ty=0$. We
write $X^h$ for the set of all real-valued lattice homomorphisms on
$X$. Clearly, $X^h$ is a subset (but not a subspace) of
$X^\sim$. Recall that the elements of $X^h$ are atoms of $X^\sim$,
see, e.g., \cite[Theorem~1.85]{Aliprantis:03}; in particular, any two
elements of $X^h$ are either disjoint or proportional. It follows that
$\varphi\in X^\sim$ is a difference of two lattice homomorphisms, if and only if
$\varphi^+$ and $\varphi^-$ are lattice homomorphisms.

\begin{proposition}\label{sublat-codim-1}
  Uniformly closed sublattices of $X$ of codimension 1 are exactly the
  kernels of differences of two lattice homomorphisms in $X'$.
\end{proposition}

\begin{proof}
  Suppose that $Y=\ker(\varphi-\psi)$ where $\varphi$ and $\psi$ are
  lattice homomorphisms.  Being the set on which $\varphi$ and $\psi$
  agree, $Y$ is a sublattice; it is uniformly closed by
  Corollary~\ref{u-subsp-codim-1}.

  Suppose that $Y$ is a uniformly closed sublattice of $X$ of
  codimension 1. By Corollary~\ref{u-subsp-codim-1}, $Y=\ker\varphi$
  for some $\varphi\in X^\sim$. It suffices to show that $\varphi^+$
  and $\varphi^-$ are lattice homomorphisms. Let $x\wedge y=0$; we
  need to show that $\varphi^+(x)\wedge\varphi^+(y)=0$ or,
  equivalently, that either $\varphi^+(x)=0$ or
  $\varphi^+(y)=0$. Suppose not, then by Riesz-Kantorovich Formulae,
  there exist $u\in[0,x]$ and $v\in[0,y]$ such that $\varphi(u)$ and
  $\varphi(v)$ are greater than zero. Then
  $\varphi(u-\lambda v)=\varphi(u)-\lambda\varphi(v)=0$ for some
  $\lambda>0$, hence $u-\lambda v\in\ker\varphi$. Since $u\perp v$ and
  $\ker\varphi$ is a sublattice, it follows that
  $u+\lambda v=\abs{u-\lambda v}\in\ker\varphi$; which is impossible.
  The proof that $\varphi^-$ is a lattice homomorphism is similar.
\end{proof}

\begin{example}
  There exist non-uniformly closed sublattices of
  codimension~1. Indeed, by Proposition~\ref{sublats-exist},
  $L_p[0,1]$ ($1\le p<\infty$) contains a sublattice of codimension
  1. It is not uniformly closed as $L_p[0,1]$ admits no real-valued
  lattice homomorphisms.

  Here is another example. Recall that linear functional $\varphi$ is
  \term{disjointness preserving} (d.p.) if $x\perp y$ implies
  $\varphi(x)=0$ or $\varphi(y)=0$. In this case, $\ker\varphi$ is a
  sublattice because if $x\in\ker\varphi$ then either $\varphi(x^+)=0$
  or $\varphi(x^-)=0$; in either case,
  $\varphi\bigl(\abs{x}\bigr)=0$. There exist d.p.
  functionals $\varphi$ that are not order bounded;
  see~\cite{Schep:16} and references there. By
  Proposition~\ref{obdd-ker-u-closed}, $\ker\varphi$ is not uniformly
  closed.
\end{example}

\begin{question}
  Let $\varphi$ be a functional on $X$.
  Propositions~\ref{obdd-ker-u-closed}, \ref{ker-full},
  \ref{ker-ideal}, and~\ref{sublat-codim-1} relate various order
  properties of $\varphi$ with those of $\ker\varphi$. However, one
  property is clearly missing from this list: it would be interesting
  to characterize when $\ker\varphi$ is a sublattice
  (not necessarily uniformly closed).

  Lemma~2 in~\cite{Abramovich:90b} or Proposition~\ref{sublat-disj-2m}
  in the current paper yield the following necessary condition: for
  every three disjoint vectors, $\varphi$ must vanish at at least one
  of them. Thus, we have the following:
  $\varphi$ is d.p. $\Rightarrow$ $\ker\varphi$ is a sublattice
  $\Rightarrow$ $\varphi$ is 3-d.p.
  % \begin{displaymath}
  %   \varphi\mbox{ is d.p.}\quad\Rightarrow\quad
  %   \ker\varphi\mbox{ is a sublattice}\quad\Rightarrow\quad
  %     \varphi\mbox{ is 3-d.p.}
  % \end{displaymath}
\end{question}

We will now show in Theorem~\ref{closures} that much of
Theorem~\ref{CK-sublat-closure} can be generalized from $C(\Omega)$ to
general vector lattices, replacing evaluation functionals with lattice
homomorphisms. In fact, most of Theorem~\ref{CK-sublat-closure} may be
deduced from Theorem~\ref{closures}. We start with some linear algebra
lemmas.

For a subset $A$ of a vector space $E$ and $n\in\mathbb N$, we write
\begin{displaymath}
  \Span_nA=\Bigl\{\alpha_1x_1+\dots+\alpha_nx_n\mid
  x_1,\dots,x_n\in A,\ \alpha_1,\dots,\alpha_n\in\mathbb R\bigr\}.
\end{displaymath}

\begin{lemma}\label{span-n}
  Let $F$ be a subspace of a vector space $E$, $A\subseteq E'$, and
  $n\in\mathbb N$. For $x\in E$,
  $x\in\bigl(\Span_nA\cap F^\perp\bigr)_\perp$ if and only if for every
  $\varphi_1,\dots,\varphi_n\in A$ there exists $y\in F$ with
  $\varphi_i(x)=\varphi_i(y)$ for $i=1,\dots,n$.
\end{lemma}

Before proving the lemma let us first demystify its statement. The set
$\Span_nA\cap F^\perp$ consists of the elements of $\Span_n A$ which
vanish on $F$, and so the first set is the maximal subspace on which
these functionals vanish. Meanwhile, the other set consists of the
vectors that allow $n$-nod interpolation by the elements of $F$ with
respect to $A$.

\begin{proof}
  Denote the sets in question by $H$ and $G$. Suppose that $x\in G$ and
  $\psi\in \Span_nA\cap F^\perp$. Then
  $\psi=\alpha_1\varphi_1+\dots+\alpha_n\varphi_n$ for some
  $\varphi_1,\dots,\varphi_n\in A$ and
  $\alpha_1,\dots,\alpha_n\in\mathbb R$. Since $x\in G$, there exists
  $y\in F$ such that $\varphi_i(x)=\varphi_i(y)$ as $i=1,\dots,n$. It
  follows that
  \begin{math}
    \psi(x)=\psi(y)=0
  \end{math}
  because $\psi\in F^\perp$. This yields  $x\in\bigl(\Span_nA\cap
  F^\perp\bigr)_\perp$. Hence $G\subseteq H$.

  Suppose $x\notin G$. Then there exist
  $\varphi_1,\dots,\varphi_n\in A$ such that $Tx\notin TF$, where
  $T\colon E\to\mathbb R^n$, $(Tx)_i=\varphi_i(x)$. It follows that
  there is a functional on $\mathbb R^n$ that separates $Tx$ from
  $TF$. That is, there exist $\bar a=(\alpha_1,\dots,\alpha_n)$ such
  that $0=\langle\bar a,Ty\rangle=\sum_{i=1}^n\alpha_i\varphi_i(y)$
  for all $y\in F$ but $\langle\bar a,Tx\rangle\ne 0$. Put
  $\psi=\sum_{i=1}^n\alpha_i\varphi_i$; then $\psi$ vanishes on $F$
  but $\psi(x)\ne 0$, so $x\notin H$.
\end{proof}

\begin{corollary}\label{span-perp}
  In the setting of Lemma~\ref{span-n}, $x\in\bigl(\Span A\cap
  F^\perp\bigr)_\perp$ if and only if for every $n$ and every
  $\varphi_1,\dots,\varphi_n\in A$ there exists $y\in F$ such that
  $\varphi_i(x)=\varphi_i(y)$ for $i=1,\dots,n$.
\end{corollary}

\begin{theorem}\label{closures}
  Let $Y$ be a sublattice of $X$. The following sets are equal:
  \begin{eqnarray*}
    Y_1&=&\Bigl\{z\in X\mid \forall\mbox{ finite }F\subseteq X^h
           \ \exists y\in Y,\mbox{ $y$ and $z$ agree on $F$}\Bigr\};\\
   Y_2&=&\Bigl\{z\in X\mid \forall h_1,h_1\in X^h\
           \exists y\in Y,\ h_1(z)=h_1(y)\mbox{ and }h_2(z)=h_2(y)\Bigr\};\\
    Y_3&=&\Bigl(\bigl(\Span X^h\bigr)\cap Y^\perp\Bigr)_\perp;\\
    Y_4&=&\Bigl(\bigl(\Span_2 X^h\bigr)\cap Y^\perp\Bigr)_\perp;\\
    Y_5&=&\Bigl(\bigl\{h_1-h_2 \mid h_1,h_2\in X^h\}
           \cap Y^\perp\Bigr)_\perp;\\
    Y_6&=&\mbox{the intersection of all uniformly closed sublattices of }X\\
      &&\qquad\mbox{of codimension at most 1 that contain $Y$}.
  \end{eqnarray*}
\end{theorem}

\begin{proof}
  It is clear that $Y_3\subseteq Y_4\subseteq Y_5$. Also, $Y_1=Y_3$
  and $Y_2=Y_4$ by Corollary~\ref{span-perp} and Lemma~\ref{span-n},
  respectively. By Proposition~\ref{sublat-codim-1}, $Y_5=Y_6$. We
  will now show that $Y_5\subseteq Y_3$; this will complete the proof.

  Suppose $z\in Y_5$. This means that if $h_1$ and $h_2$ in
  $X^h$ agree on $Y$ then they also agree at $z$. In
  particular, if $h\in X^h$ vanishes on $Y$ then
  $h(z)=0$. Suppose now that
  $f\in\bigl(\Span X^h\bigr)\cap Y^\perp$; we need to show that
  $f(z)=0$. Let $f=\alpha_1h_1+\dots+\alpha_nh_n$, where
  $h_1,\dots,h_n\in X^h$. For every $i=1,\dots,n$, if $h_i$
  vanishes on $Y$ then $h_i(z)=0$, so we may assume that no $h_i$'s
  vanishes on $Y$.

  The restriction of every $h_i$ to $Y$ is a lattice homomorphism,
  hence an atom in $Y^\sim$. Since any two atoms are either disjoint
  or proportional and non-proportional atoms are linearly independent,
  we may split the sum $\alpha_1h_1+\dots+\alpha_nh_n$ into groups
  corresponding to non-proportional atoms in $Y^\sim$; then each group
  vanishes on $Y$. Hence, without loss of generality, we may assume
  that we have only one group, i.e., that ${h_i}_{|Y}$'s are all
  proportional to some lattice homomorphism $g$ on $Y$, i.e.,
  ${h_i}_{|Y}=\beta_ig$. Since ${h_i}_{|Y}\ne 0$, we have $g\ne 0$ and
  $\beta_i\ne 0$ for all $i$.

  It follows from
  \begin{math}
    0=f_{|Y}=\sum_{i=1}^n\alpha_i{h_i}_{|Y}
    =\Bigl(\sum_{i=1}^n\alpha_i\beta_i\Bigr)g
  \end{math}
  that $\sum_{i=1}^n\alpha_i\beta_i=0$. Also, for any $i\ne j$, the
  restrictions of $\frac{1}{\beta_i}h_i$ and $\frac{1}{\beta_j}h_j$ to
  $Y$ both equal $g$, hence are equal to each other. Since $z\in Y_5$
  it follows that
  $\frac{1}{\beta_i}h_i(z)=\frac{1}{\beta_j}h_j(z)$. This means that
  $\frac{1}{\beta_i}h_i(z)$ does not depend on $i$; denote it
  $\lambda$. It follows that
  \begin{displaymath}
    f(z)=\sum_{i=1}^n\alpha_ih_i(z)=\sum_{i=1}^n\alpha_i\beta_i\lambda=0.
  \end{displaymath}
\end{proof}

Unlike Theorem~\ref{CK-sublat-closure}, the sets in
Theorem~\ref{closures} no longer equal the uniform closure of $Y$.
It would be interesting to characterize vector lattices (in
particular, Banach lattices) in which every uniformly closed
sublattice $Y$ can be written as an intersection of uniformly closed
sublattices of codimension one. In this case, the six sets in
Theorem~\ref{closures} equal $\overline{Y}$. We will show in
Theorem~\ref{codim-inters} that this is the case when $Y$ is of finite
codimension. On the other hand, this fails when $X$ is a non-atomic
order continuous Banach lattice as in this case $X$ has no closed
sublattices of finite codimension (see Corollary~\ref{nonatom-ocont}),
hence $Y_6=X$, even if $Y$ is a proper uniformly closed sublattice
(even an ideal) of $X$.

\begin{question}
  Suppose that $X^h$ separates the points of $X$. Does this imply that
  every uniformly closed sublattice can be written as an intersection
  of uniformly closed sublattices of codimension one?
\end{question}

We will now extend Corollary~\ref{CK-sublat-ideal} to general vector
latices.

\begin{lemma}\label{no-ideals}
  Suppose that $X$ has a uniformly closed sublattice of
  codimension $n$ such that it contains no ideals of $X$. Then $\dim
  X\le 2n$.
\end{lemma}

\begin{proof}
  Let $Y$ be such a sublattice. Suppose that $\dim X>2n$. Then we can
  find linearly independent $x_1,\dots,x_{2n+1}$ in $X$. Put
  $e=\bigvee_{i=1}^{2n+1}\abs{x_i}$. Then $x_i\in I_e$ as
  $i=1,\dots,2n+1$, hence $\dim I_e>2n$. Clearly, $Y\cap I_e$ is a
  sublattice of codimension at most $n$ in $I_e$, closed with respect
  to $\norm{\cdot}_e$.

  By Krein-Kakutani's Representation Theorem, we may identify $I_e$
  with a dense sublattice of $C(K)$. Let $Z$ be the closure of
  $Y\cap I_e$ in $C(K)$. By Lemma~\ref{codim-compl}, $Z$ is a closed
  sublattice of $C(K)$ of codimension at most $n$.  By
  Theorem~\ref{CK-sublat-closure},
  $Z=\bigcap_{i=1}^n\ker(\delta_{s_{2i-1}}-\alpha_i\delta_{s_{2i}})$
  for some $s_1,\dots,s_{2n}$ in $K$ and $\alpha_i\ge 0$. Put
  $J=\bigcap_{i=1}^{2n}\ker\delta_{s_i}$. Then $J$ is an ideal of
  $C(K)$ contained in $Z$. It follows that $J\cap I_e$ is an ideal in
  $I_e$, hence in $X$, contained in $Y$. By assumption,
  $J\cap I_e=\{0\}$.

  Let $t\notin\{s_1,\dots,s_{2n}\}$. Clearly, there exists $f\in C(K)$
  such that $f$ vanishes on $\{s_1,\dots,s_{2n}\}$ but not at
  $t$. Since $I_e$ is a dense sublattice of $C(K)$, we could use $Z_1$ in
  Theorem~\ref{CK-sublat-closure} to find a function $g\in I_e$
  which vanishes on $\{s_1,\dots,s_{2n}\}$ but not at $t$; this would
  contradict $J\cap I_e=\{0\}$. It follows that such a $t$ does not
  exist, i.e., $K=\{s_1,\dots,s_{2n}\}$. This yields $\dim
  I_e\le\dim C(K)\le 2n$, which is a contradiction.
\end{proof}

\begin{proposition}\label{codim-lat-ideal}
  Every uniformly closed sublattice of $X$ of codimension $n$ contains
  a uniformly closed ideal of $X$ of codimension at most $2n$.
\end{proposition}

\begin{proof}
  Let $Y$ be a uniformly closed sublattice of $X$ of codimension
  $n$. Let $J$ be the union of all the ideals of $X$ contained in
  $Y$. It is clear that $J$ is an ideal of $X$. Since the uniform
  closure of an ideal is again an ideal by Theorem~63.1
  in~\cite{Luxemburg:71}, it follows that $J$ is uniformly closed in
  $X$. Therefore, the quotient $X/J$ is Archimedean by Theorem~60.2
  in~\cite{Luxemburg:71}. Let $Q\colon X\to X/J$ be the quotient map.

  Clearly, $Q(Y)$ is a sublattice of codimension $n$ in $X/J$.
  It follows from Corollary~63.4 in~\cite{Luxemburg:71} (or can be
  easily verified directly) that $QY$ is uniformly closed in
  $X/J$. It follows from the maximality of $J$ in $Y$ that $Q(Y)$
  contains no proper non-trivial ideals: if $H$ were such an ideal
  that $Q^{-1}(H)$ would be an ideal of $X$ satisfying $J\subsetneq
  Q^{-1}(H)\subsetneq Y$, which would be a contradiction.

  It follows from Lemma~\ref{no-ideals} that $\dim(X/J)\le 2n$ and,
  therefore, $J$ has codimension at most $2n$ in $X$.
\end{proof}

\begin{remark}
  The preceding proposition has no infinite-codimensional analogue:
  the (infinite codimensional) closed sublattice of $C[-1,1]$
  consisting of all even functions contains no ideals.
\end{remark}

\begin{remark}
  The assumption that the sublattice is uniformly closed in
  Proposition~\ref{codim-lat-ideal} cannot be dropped. Indeed, by
  Proposition~\ref{sublats-exist}, $L_p[0,1]$ contains a sublattice of
  codimension one. But $L_p[0,1]$ has no ideals of codimension 2 or of
  any finite codimension; see Corollary~\ref{no-ideals-nonat}.
\end{remark}

\begin{question}
  Is the uniform closedness assumption in Lemma~\ref{no-ideals}
  necessary?
\end{question}

\begin{theorem}\label{codim-inters}
  Every uniformly closed sublattice of $X$ of codimension $n$ may be
  written as the intersection of $n$ uniformly closed sublattices of
  codimension 1.
\end{theorem}

\begin{proof}
  Let $Y$ be a uniformly closed sublattice of $X$ of codimension
  $n$. By Proposition~\ref{codim-lat-ideal}, $Y$ contains a uniformly
  closed ideal $J$ of $X$ of codimension $m\le 2n$. Let
  $Q\colon X\to X/J$ be the quotient map.  Clearly, $X/J$ is lattice
  isomorphic to $\mathbb R^m$ and $Q(Y)$ is a sublattice of $X/J$. It
  follows from Corollary~\ref{sublat-Rn} (or from
  Theorem~\ref{CK-sublat-closure}) that $Q(Y)$ may be written as the
  intersection of sublattices of codimension 1. Taking their
  preimages, we get the required representation for $Y$.
\end{proof}

Combining the theorem with Proposition~\ref{sublat-codim-1}, we get:

\begin{corollary}
  Every uniformly closed sublattice of codimension $n$ in $X$ may be
  written in the form $\bigcap_{i=1}^n\ker(\varphi_{2i-1}-\varphi_{2i})$, where
  $\varphi_1,\dots,\varphi_{2n}$ are lattice homomorphisms in $X^\sim$.
\end{corollary}

Note that some of the $\varphi_i$'s in the corollary may coincide or
equal zero.  Note also that taking $J=\bigcap_{i=1}^{2n}\ker\varphi_i$
yields a uniformly closed ideal of codimension at most $2n$ as per
Proposition~\ref{codim-lat-ideal}.

\medskip

Let $S$ and $T$ be lattice homomorphisms from
$X$ to an Archimedean vector lattice $Z$ and put $Y=\ker(S-T)$. It
follows from $Y=\{x\in X\mid Sx=Tx\}$ that $Y$ is a sublattice; it is
uniformly closed because $S-T$ is order bounded and, therefore,
uniformly continuous.

\begin{corollary}
  Uniformly closed sublattices of codimension at most $n$ in $X$ are
  exactly the kernels of differences of two lattice homomorphisms from
  $X$ to $\mathbb R^n$.
\end{corollary}

\begin{proof}
  Let $Y$ be a uniformly closed sublattice of codimension at most $n$
  in $X$. Then $Y=\bigcap_{i=1}^n\ker(\varphi_i-\psi_i)$, where
  $\varphi_i$ and $\psi_i$ are lattice homomorphisms in $X^\sim$ as
  $i=1,\dots,n$. Define $S,T\colon X\to\mathbb R^n$ via
  \begin{displaymath}
    Sx=\bigl(\varphi_1(x),\dots,\varphi_n(x)\bigr)
    \quad\mbox{and}\quad
    Tx=\bigl(\psi_1(x),\dots,\psi_n(x)\bigr).
  \end{displaymath}
  It is easy to see that $S$ and $T$ are lattice homomorphisms and
  $Y=\ker(S-T)$. The converse is trivial.
\end{proof}

\begin{question}
  Is every uniformly closed sublattice the kernel of the difference of
  two lattice homomorphisms? (Cf. Question~\ref{q:ussubs}.)
\end{question}

\bigskip

We would like to mention two results that imply that many important
classes of vector lattices lack uniformly closed sublattices of finite
codimension. Recall that a sublattice $Y$ of $X$ is \term{order dense}
if for every $0<x\in X$ there exists $y\in Y$ such that $0<y\le x$.

\begin{proposition}[\cite{Wojtowicz:98}]
  If $X$ is non-atomic then every finite codimensional sublattice is
  order dense.
\end{proposition}

\begin{proof}
  Assume that $Y$ is a finite codimensional sublattice but is not
  order dense. Then there exists $0<x\in X$ such that
  $I_x\cap Y=\{0\}$. It follows that $\dim I_x<\infty$. Being a finite
  dimensional vector lattice, $I_x$ and, therefore, $X$ admits an
  atom; a contradiction.
\end{proof}

Note that there exist proper finite codimensional uniformly closed
sublattices that are order dense. For example, consider the sublattice
of $C[0,1]$ consisting of all functions that vanish at $0$.  However,
no proper sublattice is both order closed and order dense. Recall that
if $Y$ is an order dense sublattice then $x=\sup[0,x]\cap Y$ for every
$x\in X_+$. It follows that an order dense sublattice of an order
continuous Banach lattice has to be norm dense. We can now easily
deduce the following result of \cite{Abramovich:90a}:

\begin{corollary}[\cite{Abramovich:90a}]\label{nonatom-ocont}
  Let $X$ be a non-atomic order continuous Banach lattice. Then $X$
  has no proper closed sublattices of finite codimension.
\end{corollary}

We would like to mention an important corollary of the preceding
result:

\begin{corollary}[\cite{Abramovich:90a,Wojtowicz:98}]
  If $X$ is a non-atomic order continuous Banach lattice then $X^*$
  is non-atomic.
\end{corollary}

\begin{proof}
  If $\varphi\in X^*$ is an atom, it is a lattice homomorphism, hence
  $\ker\varphi$ is a closed ideal of codimension 1 in $X$.
\end{proof}

\section{Finite codimensional ideals}

It was observed in II.5.3, Corollary~3 of~\cite{Schaefer:74} that
every ideal of co-dimension 1 in a Banach lattice is
closed. In~\cite{Abramovich:90a}, this result was extended to
finite-codimensional ideals of F-lattices. We are now going to extend
it to uniformly closed ideals.

\begin{proposition}
  Every ideal of codimension 1 is uniformly closed.
\end{proposition}

\begin{proof}
  The quotient map may be viewed as a lattice homomorphism from $X$ to
  $\mathbb R$. Now apply Corollary~\ref{u-subsp-codim-1}.
\end{proof}

The following result and example are motivated by~\cite{Abramovich:90a}.

\begin{proposition}\label{idea-fcodim-ucl}
  If $X$ is uniformly complete then every finite codimensional ideal
  in $X$ is uniformly closed.
\end{proposition}

\begin{proof}
  Let $J$ be an ideal of finite codimension $n$.  Proof is by
  induction on $n$. For $n=1$, this is the preceding
  proposition. Suppose $n>1$. We can then find an intermediate ideal
  $H$ such that $J\subsetneq H\subsetneq X$: one can take the preimage
  under the quotient map of any proper non-trivial ideal in $X/J$. By
  induction hypothesis, $J$ is uniformly closed in $H$ and $H$ is
  uniformly closed in $X$. Since uniform convergence for sequences in
  $H$ agrees with that inherited from $X$ by Proposition~2.12
  in~\cite{Taylor:20} and the remark following it, we conclude that
  $J$ is uniformly closed in $X$.
\end{proof}

It is observed in Lemma~3 of \cite{Abramovich:90a} that for every ideal
$J$ in $X$ of codimension $n$ there exists a chain of ideals
$J=J_1\subsetneq J_2\subsetneq\dots\subsetneq J_n=X$.

\begin{example}
  \emph{An ideal of codimension 2 that is not uniformly closed.} Let
  $X$ be the vector lattice of all piece-wise affine functions on
  $[0,1]$ and
  \begin{math}
    J=\bigl\{f\in X\mid\exists r>0\ f\text{ vanishes on }[0,r]\bigr\}.
  \end{math}
  Clearly, $J$ is an ideal; its codimension equals 2 as
  $X=J+\Span\{\one,f_0\}$, where $f_0(t)=t$. It can be easily verified
    that $f_0$ is in the uniform closure of $J$, hence $J$ is
    not uniformly closed.

    Let $Y=\bigl\{f\in X\mid f(0)=0\bigr\}$. Then
    $J\subsetneq Y\subsetneq X$ is a chain of ideals, hence $J$ is
    uniformly closed in $Y$ and $Y$ is uniformly closed in $X$, yet
    $J$ is not uniformly closed in $X$.
\end{example}

Since a subset of a Banach lattice is closed if and only if it is
uniformly closed, Proposition~\ref{idea-fcodim-ucl} and
Corollary~\ref{nonatom-ocont} easily yield the following two results
from~\cite{Abramovich:90a}:

\begin{corollary}[\cite{Abramovich:90a}]
  Every finite-codimensional ideal in a Banach lattice is closed.
\end{corollary}

\begin{corollary}[\cite{Abramovich:90a}]\label{no-ideals-nonat}
  A non-atomic order continuous Banach lattice admits no proper ideals
  of finite codimension.
\end{corollary}

Let $J$ be a closed ideal of codimension $n$ in $C(\Omega)$. By
Theorem~\ref{CK-sublat-closure} and Corollary~\ref{closed-ideal},
there exist distinct $n$ points $t_1,\dots,t_n$ such that $f\in J$ if
and only if $f$ vanishes at these points. That is,
$J=\bigcap_{i=1}^n\ker\delta_{t_i}$. Note that $\ker\delta_{t_i}$ is an
ideal of codimension 1. This is generalized in the following
proposition that we borrow from \cite{Abramovich:90a,Abramovich:90b};
we include it for the sake of completeness.

\begin{proposition}
  Let $J$ be a uniformly closed ideal of $X$ of codimension $n$. There
  exist unique (up to a permutation) ideals $H_1,\dots,H_n$ of
  codimension $1$ such that $J=\bigcap_{i=1}^nH_i$.
\end{proposition}

\begin{proof}
  Consider the quotient map $Q\colon X\to X/J$.
  Since $J$ is uniformly closed, $X/J$ is an Archimedean vector
  lattice of dimension $n$, hence is lattice isomorphic to
  $\mathbb R^n$. Therefore, there are exactly $n$ distinct ideals
  $E_1,\dots,E_n$ of codimension $1$ in $X/J$ and
  $\bigcap_{i=1}^nE_i=\{0\}$. Now take $H_i=Q^{-1}(E_i)$ as
  $i=1,\dots,n$.

  To prove uniqueness, let $H$ be an ideal of codimension 1 in $X$ such
  that $J\subseteq H$. It is easy to see that $Q(H)$ is an ideal of
  codimension 1 in $X/J$, hence $H=Q^{-1}(E_i)$ for some $i$.
\end{proof}

\medskip

We would like to thank the reviewers for valuable comments and suggestions.

\section*{Funding}
The second author was supported by NSERC.

\end{document}